\documentclass[12pt]{article}
\usepackage[left=1in, right=1in, top=1in, bottom=1in]{geometry}
\usepackage[font=small,labelfont=bf,width=0.8\textwidth,skip=16pt]{caption}
\usepackage{graphicx}
\usepackage{here} 
\usepackage{bbm} 
\usepackage{color,latexsym,amsfonts,amssymb,amsmath,mathbbol,bm,hyperref}
\usepackage{amsthm}
\usepackage{verbatim}
\usepackage{mathrsfs}
\usepackage{fancyhdr}
\usepackage{bbold}

\numberwithin{figure}{section}
\numberwithin{table}{section}
\numberwithin{equation}{section}

\newcommand{\E}{\mathbb E}

\usepackage{amsthm}
\newtheoremstyle{mythm}{18pt}{0pt}{\itshape}{}{\bfseries}{.}{12pt}{}
\newtheoremstyle{mydefn}{18pt}{0pt}{}{}{\bfseries}{.}{12pt}{}
\theoremstyle{mythm}
\newtheorem{theorem}{Theorem}[section]
\newtheorem{lemma}[theorem]{Lemma}

\theoremstyle{mydefn}
\newtheorem{remark}[theorem]{Remark}

\def\Pr{\mathbb{P}}

\newcommand{\Pn}{{\rm Poisson}}

\def\Ref#1{(\ref{#1})}
\parskip = \baselineskip
\newcommand{\eqs}{\begin{eqnarray*}}
\newcommand{\ens}{\end{eqnarray*}}

\def\l{\lambda}

\newcommand{\beas}{\begin{eqnarray*}}
\newcommand{\enas}{\end{eqnarray*}}

\newcommand{\eqa}{\begin{eqnarray}}
\newcommand{\ena}{\end{eqnarray}}
\newcommand{\eq}{\begin{equation}}
\newcommand{\en}{\end{equation}}

\def\ignore#1{}

\def\Po{{\rm Pn}}

\def\l{\lambda}

\def\ed{\stackrel d=}

\makeatletter
\renewcommand\theequation{\thesection.\@arabic\c@equation}
\makeatother

\usepackage{color,graphicx,mathbbol, enumerate}

\def\ind{\bm{1}}  
\def\E{\mathbb{E}} 

\def\Pr{\mathbb{P}}

\def\Sk{\rm {Sk}}
\def\tP{\mathbf{P}}
\def\tN{\mathbf{N}}
\def\tX{\mathbf{X}}
\def\tf{\mathbf{f}}
\def\tY{\mathbf{Y}}
\def\hU{\hat U}
\def\hV{\hat V}
\def\hUi{\hat U^{(i)}}
\def\hVi{\hat V^{(i)}}
\def\sumin{\sum_{i=1}^n}

\def\Ref#1{(\ref{#1})}

\allowdisplaybreaks

\begin{document}
\parindent 0cm
\parskip .5cm
\title{Approximation of the difference of two Poisson-like counts by Skellam}
\author{H.~L.~Gan\footnote{Mathematics Department,
Northwestern University,
2033 Sheridan Road,
Evanston, IL 60208, USA. E-mail: ganhl@math.northwestern.edu} \and Eric D.~Kolaczyk\footnote{Department of Mathematics and Statistics,
Boston University,
111 Cummington Mall,
Boston, MA 02215, USA. E-mail: kolaczyk@bu.edu}}
\maketitle
\begin{abstract}
Poisson-like behavior for event count data is ubiquitous in nature.  At the same time, differencing of such counts arises in the course of data processing in a variety of areas of application.  As a result, the Skellam distribution -- defined as the distribution of the difference of two independent Poisson random variables -- is a natural candidate for approximating the difference of Poisson-like event counts.  However, in many contexts strict independence, whether between counts or among events within counts, is not a tenable assumption.  Here we characterize the accuracy in approximating the difference of Poisson-like counts by a Skellam random variable.  Our results fully generalize existing, more limited results in this direction and, at the same time, our derivations are significantly more concise and elegant.  We illustrate the potential impact of these results in the context of problems from network analysis and image processing, where various forms of weak dependence can be expected.
\end{abstract}


\section{Introduction}
Given two independent Poisson random variables $X$ and $Y$ with means $\l_1$ and $\l_2$, the Skellam distribution, originally attributed to~\cite{S46}, is defined as the distribution of the difference of $X$ and $Y$.  Formally, a random variable $W$ defined on the integers is said to have Skellam distribution with parameters $\l_1, \l_2>0$, which we will denote by $\text{Sk}(\l_1, \l_2)$, if for all $k \in \mathbb{Z}$,
\begin{align}
\Pr(W = k) = e^{-(\l_1 + \l_2)} \left( \sqrt{\frac{\l_1}{\l_2}}\right)^k I_k(2 \sqrt{\l_1\l_2}),
\end{align}
where $I_k(2 \sqrt{\l_1\l_2})$ denotes the modified Bessel function of the first kind with index $k$ and argument $2 \sqrt{\l_1\l_2}$.

In light of the ubiquity of Poisson-like behavior in nature and the ease with which differencing can arise in data processing, it is perhaps no surprise that the Skellam distribution has seen use in a variety of areas of application.  These include application to neural decoding in computational neuroscience~\cite{SH10}, denoising~\cite{HW12} and edge detection~\cite{HKK07} in image processing, conservation laws in particle physics (e.g., \cite{BKS13,MFRS13}), x-ray fluoroscopy in radiology~\cite{CBCRP13}, and the identification of genetic variants in bioinformatics~\cite{AWQ11}.  Most recently, the Skellam distribution has been found to have a role in network analysis~\cite{BKV14}.

In each of these contexts, there are two categories of events being counted and the resulting sums (i.e., denoted $X$ and $Y$ above) are then differenced.  The counting, of course, motivates use of the Poisson distribution in modeling.  The events being counted might be the spiking of neurons in two areas of the brain, the arrival of particles in two adjacent detectors in an array, the genetic variants in two nearby regions of the genome, or the presence / absence of a given subgraph across subsets of nodes in a network.  Ideally, indicators of these events are independent, both within each type of event category and across the two categories.  Independence within is ideal for arguing a Poisson approximation to the counts in each of the two event categories (i.e., in arguing Poisson approximations to the distributions of each of $X$ and $Y$).  At the same time, strictly speaking, independence across the two categories would seem to be necessary, as it is inherent to the definition of the Skellam distribution (i.e., the distribution of $X-Y$).

However, just as it is known that a Poisson approximation to event counts can be accurate under various forms of weak dependence, it is natural to expect that the \emph{difference} of Poisson-like counts might be similarly well-approximated under some form of weak dependence.  If the events are \emph{dependent within} each category but \emph{independent between} categories, then formal results of this nature follow from trivial extension of existing results for Poisson approximation.  On the other hand, if events are \emph{dependent between} categories, then such results are not immediate.

Motivated by the problem of subgraph counting in noisy networks, where it was noted that such complex dependencies can arise easily, \cite{BKV14} initiated work on such a general Skellam approximation using Stein's method.  However, the results provided in~\cite{BKV14} are limited, in that the bounds for the Stein factors therein were derived using a purely analytic approach for the Kolmogorov metric and were restricted to the case where $\l_1 = \l_2$. In pursuing the same problem of general Skellam approximation here, also using Stein's method, our approach in this paper will use the so-called probabilistic method by exploiting properties of generators of Markov processes, in contrast to the direct analytic approach used in \cite{BKV14}. The main advantages of our approach here are that we can derive bounds for the more general case $\lambda_1 \neq \lambda_2$, and that the proofs via this approach are significantly easier to derive.

The importance of our work is fundamental in nature, yet it has the potential to be wide-ranging in its practical impact.  In each of the application domains described above there is the very real possibility of general weak dependence among event counts (i.e., both within and between categories).  For example, dependencies arise naturally when counting subgraphs in noisy networks, either through dependency in the measurements underlying the construction of the network in the first place or through overlap of vertex subsets while counting~\cite{BKV14}.  Alternatively, dependency can be expected in particle counts obtained by the types of charge-coupled device (CCD) imaging instruments commonly used in astrophysics, due to so-called spillover effects (e.g., \cite{TBGN00,ASNG15}).  We will expand more on both of these examples in a later section.

There is by now, of course, a large and rich literature on the use of Stein's method to characterize accuracy of Poisson approximation to event counts, see~\cite{BHJ92} for a monograph on the topic.  However, the focus of this paper is on approximating the \emph{difference} of two Poisson-like counts, which to our knowledge is yet to be studied in depth other than the work of~\cite{BKV14}.  In~\cite{BKV14}, the focus was on approximating the distribution of what were termed `noisy' subgraph counts, i.e., subgraph counts in graphs wherein our knowledge of the presence / absence status of edges among vertex pairs is uncertain.  There the focus was on a centered version of such counts, which was found upon manipulation to yield a difference of two Poisson-like sums and, hence, motivated approximation by Skellam.  We use a simple version of the same type of problem as one of two illustrations of our results later in this paper.  Nevertheless, as also pointed out by~\cite{BKV14}, the use of Stein's method for noisy graphs is different from that used traditionally for random graphs.  Stein's method was first introduced to approximation theory for random graphs in~\cite{B82}, wherein both Poisson and Normal approximation results for isolated trees in random graphs were derived. The results for the Normal case were expanded in~\cite{BKR89} to a variety of applications such as subgraph counts and the number of isolated vertices. For summaries of Stein's method results for random graphs, see~\cite{BHJ92, JLR11}, particularly the former for Poisson approximation results that are more relevant to the work in this paper for obvious reasons.  

This paper is organized as follows. In Section~2 we construct our framework for Stein's method for the Skellam distribution, and derive bounds for the relevant Stein factors. In Section~3 we utilise each of these bounds in two example applications: counting subgraphs in noisy networks and counting particles in imaging. Both examples are relatively simple but can be easily generalised. The paper concludes with a discussion of our results and some open questions in Section~4.


\section{Stein's method for the Skellam distribution}

Our results are derived using an adaptation of multivariate Poisson approximation. While the Skellam distribution is univariate, the objects we are typically interested in approximating with the Skellam distribution are differences of two random variables. Our approach reflects this by initially considering bivariate Poisson approximation and then choosing test functions that project down to the univariate case appropriately.

We begin by noting the bivariate Poisson Stein identity. Note that $(X,Y)$ are said to be bivariate Poisson with parameters $(\l_1, \l_2)$ if $X$ and $Y$ have marginal distributions $\Pn(\l_1)$ and $\Pn(\l_2)$ and are independent.
\begin{lemma}[Bivariate Poisson Stein identity]
$(X,Y)$ is a bivariate Poisson distributed random vector with parameters $(\l_1, \l_2)$ if and only if for all functions $h$ in a family of suitable functions, $\E(\mathcal A h(X,Y)) = 0$, where
\begin{align}
\mathcal A h(x,y) &= \lambda_1[h(x+1, y) - h(x,y)] + x[h(x-1, y) - h(x,y)]\notag\\
	&\ \ \ + \lambda_2[h(x,y+1) - h(x,y)] + y[h(x, y-1) - h(x,y)].\label{2dimgen}
\end{align}
\end{lemma}
Details about multivariate Poisson approximation via Stein's method can be found in \cite{B88, B05}. For the Skellam distribution, we seek to modify bivariate Poisson approximation by considering test functions that depend only upon the difference between $X$ and $Y$. Noting that we will be abusing notation slightly by often writing bivariate functions that depend only upon the difference as a univariate function, for example $f(x,y) = f(x-y)$, for any function $f$ we define the Stein equation where we set  $h_f(x,y) =: h(x,y)$ by
%
\begin{align}
\mathcal A h_f(x,y) = f(x-y) - \text{Sk}(\l_1, \l_2)\{f\},\label{Skellamsteineq}
\end{align}
where $\text{Sk}(\l_1, \l_2)\{f\} := \E f(Z)$ and $Z \ed \text{Skellam}(\l_1,\l_2)$. Hence by taking expectations it is sufficient to find a uniform bound for $\E \mathcal A h_f(X,Y)$ to bound $\E f(X-Y) - \text{Sk}(\l_1,\l_2)\{f\}$ for any $f$. We will consider all $f$ from the family of test functions corresponding to indicator functions on the difference of the two coordinates, which encapsulates total variation distance. That is $\mathcal{F}_{TV} = \{f: f(x,y) = \bm{1}_A(x-y), A \subset \mathbb Z \}$.

Let $\Delta_i h(x,y) = h( (x,y) + \mathbf{e}^{(i)}) - h(x,y)$ where $\mathbf{e}^{(i)}$ denotes a unit vector in coordinate $i$ for $i \in \{1,2\}$. Also let $\Delta^2_{ij}h(x,y) = \Delta_i(\Delta_j h(x,y))$ where $j \in \{1,2\}$ also. To apply Stein's method successfully, bounds of the right order are required for the Stein factors,
\begin{align*}
\| \Delta_{i}h \| &= \sup_{f \in \mathcal{F}_{TV}}\sup_{x,y} | \Delta_{i} h(x,y) | \\
\| \Delta_{ij}^2 h\| &:= \sup_{f \in \mathcal{F}_{TV}} \sup_{x,y}| \Delta_{ij}^2 h(x,y)|\\ 
&= \sup_{f \in \mathcal{F}_{TV}} \sup_{x,y}\left| h((x,y)+\mathbf{e}^{(i)} + \mathbf{e}^{(j)})  - h((x,y)+ \mathbf{e}^{(i)}) - h((x,y)+ \mathbf{e}^{(j)}) + h(x,y) \right|.
\end{align*}

\begin{theorem}\label{steinfactorthm} 
For $i, j \in \{1,2\}$, 
\begin{align}
\| \Delta_i h \| &\leq \min \left\{ 1, \sqrt{ \frac{2}{e \cdot \max\{\l_1, \l_2\}}}\right\},\label{firstdiff}\\
\| \Delta_{ij}^2h\| &\leq \min \left\{ 1,  \frac{1}{2\max\{\l_1,\l_2\}^2} + \frac{\sqrt2\log^+(\sqrt2 \max\{\l_1,\l_2\})}{\max\{\l_1,\l_2\}}\right\},\label{seconddiff}
\end{align}
where $\log^+(x) = \max\{\log(x), 0\}$. Furthermore,
\begin{align}
\| \Delta_i h\| \leq \int_0^\infty e^{-t} \max\left\{ 1, e^{-(\l_1 + \l_2)(1-e^{-t})} I_0((\l_1 + \l_2)(1-e^{-t})) \right\} dt \sim \sqrt{\frac{2}{\pi(\l_1 + \l_2)}}, \label{firstdiff2}
\end{align}
where the asymptotic equivalence is for when both $\l_1, \l_2$ are large.
\end{theorem}


\begin{proof}
Our proof will follow similar ideas and techniques used in univariate Poisson approximation, for example Lemma 10.2.5~in \cite{BHJ92}. Note that we will prove the bounds in the case where $i = j = 1$, and the other cases follow essentially the same proof and hence are not included.

It can be shown that for any bounded function $f$, the (well-defined) solution to the Stein equation~\Ref{Skellamsteineq} is
\begin{align} h_f(x,y) = - \int_0^\infty\big[ \E f(Z_{x,y}(t)) - \text{Sk}(\l_1,\l_2)\{f\} \big]dt, \label{steinsol}\end{align}
where $Z_{x,y}(t)$ is a Markov process starting at $(x,y)$ and following generator~\Ref{2dimgen}. Hence,
\begin{align} 
\Delta_1 h(x,y) &= - \int_0^\infty \E \left[ f(Z_{x+1,y}(t)) - f(Z_{x,y}(t)) \right] dt\label{Dx},\\
\Delta_{11}^2h(x,y) &= - \int_0^\infty \E\left[ f(Z_{x+2,y}(t)) - f(Z_{x+1,y}(t)) - f(Z_{x+1,y}(t)) + f(Z_{x,y}(t)) \right]dt.\label{Dxx}
\end{align}
We will construct couplings by defining the following \emph{independent} processes:

\begin{tabular}{c l}
$D_1(t)$ & A pure death process with rate 1 and $D_1(0)=1$,\\
$D_2(t)$ & A pure death process with rate 1 and $D_2(0)=1$,\\
$D_x(t)$ & A pure death process with unit-per-capita death rate and $D_x(0) = x$,\\
$D_y(t)$ & A pure death process with unit-per-capita death rate and $D_y(0) = y$,\\
$Z^{\l_1}_0(t)$ & An immigration-death process with immigration rate $\l_1$,\\
	& unit-per-capita death rate and $Z^{\l_1}_0(0) = 0$,\\
$Z^{\l_2}_0(t)$  &An immigration-death process with immigration rate $\l_2$,\\
	& unit-per-capita death rate and $Z^{\l_2}_0(0) = 0$.\\
\end{tabular} 

We can define a coupling (see Theorem~2.1 of~\cite{B05} for more details) such that
\begin{align*}
Z_{x+1,y}(t) &:= Z_{x,y}(t) + (D_1(t),0),\\
Z_{x,y}(t) &:= (Z_0^{\l_1}(t),0) + (0,Z_0^{\l_2}(t)) + (D_x(t), 0) + (0, D_y(t)).
\end{align*}
Using this coupling, \Ref{Dx} now becomes
\begin{align*}
\Delta_1 h(x,y) = -\int _0^\infty e^{-t} &\E\big[ f(Z_0^{\l_1}(t) + D_x(t) - Z_0^{\l_2}(t) - D_y(t) + D_1(t))\\
	& - f(Z_0^{\l_1}(t) + D_x(t)  -Z_0^{\l_2}(t) - D_y(t))\big| \ind_{D_1(t)=1}\big] dt.
\end{align*}
Note that if $D_1(t)=0$ then the two terms in the expectation cancel out. Given $f \in \mathcal{F}_{TV}$, then as $f$ is either $0$ or $1$, the constant bound is immediate. For the $(\l_1, \l_2)$ dependent bound, the term in the expectation can be evaluated as
\begin{align}
&\sum_{k=0}^\infty \Big\{ \E[ f(k + D_x(t) - Z_0^{\l_2}(t) - D_y(t) + 1) | Z_0^{\l_1}(t) = k] \Pr(Z_0^{\l_1}(t)  = k\notag)\\
	&\ \ \ - \E[ f(k +1 + D_x(t) - Z_0^{\l_2}(t) - D_y(t)) | Z_0^{\l_1}(t) = k+1] \Pr(Z_0^{\l_1}(t) = k+1)\Big \}\notag \\
	&\ \ \ -  f(D_x(t) - Z_0^{\l_2}(t) - D_y(t)) \Pr(Z_0^{\l_1}(t) = 0\notag)\\
	&= \sum_{k=0}^\infty\E[ f(k +1 + D_x(t) - Z_0^{\l_2}(t) - D_y(t))] (\Pr(Z_0^{\l_1}(t) = k) - \Pr(Z_0^{\l_1}(t) = k+1)) \notag\\
		&\ \ \ -  f(D_x(t) - Z_0^{\l_2}(t) - D_y(t)) \Pr(Z_0^{\l_1}(t) = 0).\label{fbits}
\end{align}
Noting that it can be shown that $Z_0^{\l_1}(t) \ed \Po(\l_1(1-e^{-t}))$ (page~101 of~\cite{HPS86}), the above can be bounded by
\begin{align}
\sum_{k=0}^\infty& | \Pr(Z_0^{\l_1}(t) = k+1) - \Pr(Z_0^{\l_1}(t) = k)| + \Pr(Z_0^{\l_1}(t)= 0) \notag \\
&= 2 \max_{x\geq0}  \Pr(Z_0^{\l_1}(t) = x) \leq 2 \cdot \frac{1}{\sqrt{2e\l_1(1-e^{-t})}}\label{maxbound},
\end{align}
where the final bound on Poisson probabilities can be found in~\cite{BHJ92} (A.2.7). Recall that the functions $f$ under consideration are indicator functions on the real line. Now given that each of the first differences of the Poisson probabilities is multiplied by $f$ in~\Ref{fbits}, then the worst case for the function $f$ would be to include either all the positive or negative differences from $\Pr(Z_0^{\l_1}(t) = k) - \Pr(Z_0^{\l_1}(t) = k+1)$. As the bound in~\Ref{maxbound} contains both the positive and negative differences, we can drop a factor of 2 in our final bound.
\begin{align*}
| \Delta_1h(x,y) | &\leq \int_0^\infty e^{-t} \min\left\{ 1, \frac{1}{\sqrt{2e\l_1(1-e^{-t})}} \right\} dt\\
	&= \int_0^{-\log(1-\frac{1}{2e\l_1})} e^{-t} dt + \int_{-\log(1-\frac{1}{2e\l_1})}^\infty \frac{e^{-t}}{\sqrt{2e\l_1(1-e^{-t})}}dt
		= \sqrt{\frac{2}{e\l_1}} - \frac{1}{2e\l_1}.
\end{align*} 
The final result in \Ref{firstdiff} is achieved by noting that instead of conditioning upon $Z_0^{\l_1}(t)$ we could equally have conditioned upon $Z_0^{\l_2}(t)$ with the same corresponding final result.

For the second bound \Ref{firstdiff2}, instead of conditioning upon only $Z_0^{\l_1}(t)$, we will condition on both $Z_0^{\l_1}(t)$ and $Z_0^{\l_2}(t)$. Therefore similarly to earlier we need to bound
\begin{align}
\sum_{k=0}^\infty&\E[ f(k +1 + D_x(t) - D_y(t))] \left(\Pr(Z_0^{\l_1}(t) - Z_0^{\l_2}(t) = k) - \Pr(Z_0^{\l_1}(t) - Z_0^{\l_2}(t)= k+1)\right)\notag  \\
		&\ \ \ -  f(D_x(t) - D_y(t)) \Pr(Z_0^{\l_1}(t) -Z_0^{\l_2}(t) = 0),\label{twodiffs}
\end{align}
and hence we need a suitable bound for $\max_k \{\Pr(Z_0^{\l_1}(t) - Z_0^{\l_2}(t)=k)\}$. Recalling the distributions of $Z_0^{\l_1}(t)$ and $Z_0^{\l_2}(t)$, this boils down to finding a uniform bound for the maximum of a Skellam distribution. Using the characteristic function inversion formula,
\begin{align*} \text{Sk}(\l_1, \l_2)\{k\} &= \frac{1}{2\pi} \int_{-\pi}^\pi e^{-itk} e^{\l_1(e^{it} - 1)} e^{\l_2(e^{-it} - 1)} dt\\
	&\leq \frac{1}{2\pi} \int_{-\pi}^\pi e^{(\l_1 + \l_2)(\cos t - 1)} dt = e^{-(\l_1 + \l_2)}I_0(\l_1 + \l_2),
\end{align*}
where the last equality follows from~\cite{AS64} (9.6.19). The final bound in the theorem is now clear by starting with~\Ref{twodiffs}, following the same argument as for the bound which only depended upon $\l_1$, and then where a bound is required for $\max_{x \geq 0} \Pr(Z_0^{\l_1}(t) = x)$ in the earlier argument in~\Ref{maxbound}, use the above Skellam bound.  The asymptotic result can be derived from the fact that $I_0(z) \sim \frac{e^z}{\sqrt{2\pi z}}$ from~\cite{AS64} (9.7.1).

The bounds for the second difference are derived in a similar manner.
\begin{align}
\Delta_{11}^2h(x,y) = -\int_0^\infty e^{-2t}\E \Big[& f(Z_0^{\l_1}(t) - Z_0^{\l_2}(t) + D_x(t) - D_y(t) + D_1(t) + D_2(t))\notag\\
	& - f(Z_0^{\l_1}(t) - Z_0^{\l_2}(t) + D_x(t) - D_y(t)+ D_1(t))\notag\\
	& - f(Z_0^{\l_1}(t) - Z_0^{\l_2}(t) + D_x(t) - D_y(t)+ D_2(t))\notag\\
	& +f(Z_0^{\l_1}(t) - Z_0^{\l_2}(t) + D_x(t) - D_y(t))\Big| \ind_{D_1(t)=D_2(t)=1}\Big] dt.\label{part1}
\end{align}
Similarly to earlier, we have conditioned upon $D_1(t) =D_2(t) = 1$ in the above equation. Note that as $f\in \mathcal{F}_{TV}$, we can bound the expectation in the integral by 2. This immediately gives the first of the two bounds in the theorem.

We now work on a $(\l_1, \l_2)$ dependent bound in a similar fashion as for the first difference. Without loss of generality, assume that $\l_1 \geq \l_2$. The term in the expectation can be evaluated as follows,
\begin{align}
\sum_{k=-2}^\infty &\Big\{  \E[ f(k- Z_0^{\l_2}(t)+D_x(t) - D_y(t) + 2) | Z_0^{\l_1}(t) = k] \Pr (Z_0^{\l_1}(t) = k)\notag\\
	& - 2 \E [f(k+1-Z_0^{\l_2}(t) + D_x(t) - D_y(t) + 1) |  Z_0^{\l_1}(t) = k+1]\Pr(Z_0^{\l_1}(t)= k+1)\notag\\
	& + \E [f(k+2-Z_0^{\l_2}(t)+ D_x(t) - D_y(t) ) |  Z_0^{\l_1}(t) = k+2]\Pr(Z_0^{\l_1}(t) = k+2)\Big\}\notag\\
	&= \sum_{k=-2}^\infty \Big\{\E[ f(k+2 -Z_0^{\l_2}(t)+ D_x(t) - D_y(t))]\notag \\
	& \  \cdot(\Pr (Z_0^{\l_1}(t)  = k) - 2 \Pr (Z_0^{\l_1}(t) = k+1) + \Pr (Z_0^{\l_1}(t) = k+2)) \Big\}\label{part2}.
\end{align}
Note that for bounding $\Delta_{12}^2h(x,y)$ we modify this approach by conditioning on $Z_0^{\l_1}(t)$ being equal to $k, k-1, k+1, k$ respectively for the four terms in \Ref{part1}. The other cases follow by symmetry. Given $|f(x)| \leq 1$, the absolute value of the above is bounded by
\[ \sum_{k=-2}^\infty |\Pr (Z_0^{\l_1}(t) = k) - 2 \Pr (Z_0^{\l_1}(t)= k+1) + \Pr (Z_0^{\l_1}(t)= k+2)|, \]
which has a natural bound of 2. Recalling $Z_0^{\l_1}(t) \ed \Po(\l_1(1-e^{-t}))$, the above becomes a sum of second differences of Poisson probabilities. For $X \ed \Po(\lambda)$, 
\begin{align*}
\sum_{k=0}^\infty | p_k - 2p_{k-1} + p_{k-2}| &= \frac{1}{\l^2} \sum_{k=0}^\infty p_k | \l^2 - 2k\l + k(k-1)|\\
	&= \frac{1}{\l^2} \E | \l^2 - 2X\l + X(X-1)|\\
	&\leq \frac{1}{\l^2} \sqrt{ \E\left[ ( \l^2 - 2X\l + X(X-1))\right]^2 }\\
	&= \frac{\sqrt2}{\l},
\end{align*}
where the inequality is from H{\"o}lder's inequality. If $\max\{\l_1,\l_2\} < \frac{1}{\sqrt 2}$, then we achieve the constant bound in \Ref{seconddiff}, so assuming $\max\{\l_1,\l_2\} \geq \frac{1}{\sqrt 2}$, this gives
\begin{align*}
\| \Delta_{11}^2 h \| &\leq \int_0^\infty e^{-2t} \cdot \min\left\{2, \frac{\sqrt2}{\max\{\l_1, \l_2\} (1-e^{-t})}\right\} dt\\
&= \frac{1}{2\max\{\l_1,\l_2\}^2} + \frac{\sqrt2\log(\sqrt2 \max\{\l_1,\l_2\})}{\max\{\l_1,\l_2\}}.
\end{align*}
\end{proof}

\begin{remark}\label{plusbound}
Noting that $\frac12(\l_1+\l_2) \leq \max\{ \l_1, \l_2\} \leq \l_1 + \l_2$, we can replace the maximum terms in \Ref{firstdiff} and \Ref{seconddiff} with the following more aesthetically pleasing but less sharp bounds.
\begin{align}
\| \Delta_i h \| &\leq \min \left\{ 1, \sqrt{ \frac{4}{e (\l_1+\l_2)}}\right\},\label{firstdiff2a}\\
\| \Delta_{ij}^2h\| &\leq \min \left\{ 1,  \frac{2}{(\l_1+\l_2)^2} + \frac{2\sqrt2\log^+(\sqrt2 (\l_1+\l_2)\})}{\l_1 + \l_2}\right\},\label{seconddiff2}
\end{align}
\end{remark}

\section{Applications}

We illustrate the use of our results on approximation by Skellam through two applications.  Each is a simple caricature of a more complicated application in which such approximation has been explored in the context of a specific real application.  The first pertains to the problem of subgraph counts in noisy networks, as introduced in~\cite{BKV14}, while the second relates to photon counting devices in image processing.

\subsection{Measurement errors in Erd\H{o}s-R\'enyi graph edge counts}

The analysis of network data is widespread across the scientific disciplines (e.g., \cite{J10,EDK09,N10}).  In applied network analysis, a common {\it modus operandi} is to (i) gather basic measurements relevant to the interactions among elements in a system of interest, (ii) construct a graph-based representation of that system, with nodes serving as elements and links indicating interactions between pairs of elements, and (iii) summarize the structure of the resulting graph using a variety of numerical and visual tools.  See~\cite[Chs 3 \& 4]{EDK09} for background and several case studies illustrating this process.  Key here is the point that the process of network analysis usually rests upon some collection of measurements of a more basic nature and there are usually errors inherent in those measurements.   Unfortunately, the uncertainty in approximating some true graph $G=(V,E)$ by some estimated graph $\hat{G}=(V,\hat{E})$, which manifests as errors in our knowledge of the presence/absence of edges between vertex pairs, must necessarily propagate to any estimates of network summaries $\eta(G)$ we seek.  Yet currently there is little in the literature by way of formal and principled statistical methodology for dealing with this propagation of error.  A natural first step in this direction is a distributional analysis.  

This problem was first formalized in~\cite{BKV14}, where the focus was on the distribution of subgraph count statistics in noisy networks.  And, since it is standard in the applied network analysis literature to cite observed subgraph counts, the quantity studied in~\cite{BKV14} was the discrepancy between observed and true subgraph counts.  Particular emphasis was placed on the simplest case where the subgraph of interest is an edge, and the corresponding subgraph count, the total number of edges.  The statistic of interest therefore was the discrepancy $D=|E| - |\hat{E}|$.  Accordingly, we consider the same statistic here, but in the specific case where the true underlying graph $G$ is a classical random graph. 

Formally, suppose that $G$ is an Erd\H{o}s-R\'enyi random graph with $n$ possible edges (i.e., for notational simplicity, $n$ refers to the number of vertex pairs rather than the number of vertices).  This graph is not necessarily a complete graph, but rather each vertex pair has an edge independently with probability $p_i$. We will denote by $U_i, i \in \{ 1,\ldots, n\}$, the indicator random variable such that $U_i=1$ if an edge exists between the $i$-th vertex pair. 

Motivated by the discussion above, suppose instead of observing the true graph $G$, we instead observe a version $\hat{G}$ with errors. Let $V_i, i \in \{ 1,\ldots, n\}$, be the associated edge indicator variable for the observed graph and furthermore set the conditionally independent error probabilities to be
\begin{align*}
\Pr(V_i = 0 | U_i = 1) &= r_i,\\
\Pr(V_i = 1 | U_i = 0) &= s_i.
\end{align*}
In this setup, let $U = \sum_{i=1}^n U_i$, $V = \sum_{i=1}^n V_i$, $V_i$ is independent of $U_j, j \neq i$. In this case, $U - V$ would therefore represent the difference in the number of edges of each graph. That is, $U-V = |E| - |\hat{E}|$.  We will aim to explicitly quantify the accuracy of a Skellam approximation for $U-V$.

The details of our problem statement differ slightly from that of~\cite{BKV14}, in that the true underlying graph $G$ is random, but the spirit remains the same, in that the discrepancy $D$ is the difference of two random variables $U$ and $V$ that are certainly not independent.  Furthermore, and a significant departure from~\cite{BKV14}, we do not require that $\E[U-V]=0$.  Leveraging the main result of this paper, we have the following.
\begin{theorem}
In the above setup, if we set $\l_1 = \sumin r_ip_i$ and $\l_2 = \sumin s_i(1-p_i)$, then
\begin{align}
 &d_{TV}(\mathcal{L}(U-V), \text{Sk}(\l_1,\l_2)) \notag\\
 &\leq \sumin (p_ir_i + (1-p_i)s_i)^2 \left[ \frac{2}{\left[\sumin (p_ir_i + (1-p_i)s_i)\right]^2} + \frac{ 2\sqrt 2 \log( \sqrt 2\sumin (p_ir_i + (1-p_i)s_i))}{\sumin (p_ir_i + (1-p_i)s_i)}\right]. \label{ex1bound}
 \end{align}
\end{theorem}

\begin{proof}
The first thing to note that is while we are trying to estimate the difference of $U$ and $V$, we do not need to consider edges that exist in both random graphs. Let $\hU$ denote the number of edges that are in the true graph but not the observed graph, and similarly let $\hV$ be the number of edges that are not in the true graph but are in the observed graph. In this fashion, $U - V = \hU - \hV$. (As an aside, we note that in~\cite{BKV14} the problem is necessarily formulated directly in terms of what we refer to as $\hU-\hV$, since there the true graph $G$ is assumed nonrandom.)  We similarly define $\hU_i$ and $\hV_i$ as indicators for individual edges, note that $\Pr(\hU_i = 1)=p_ir_i$ and $\Pr(\hV_i=1) = (1-p_i)s_i$. 
We are required to bound 
\begin{align}
&\E \mathcal A h(\hU,\hV)= \E \left[ \sum_{i=1}^n\left[ p_ir_i(h(\hU+1,\hV) - h(\hU,\hV))\right] + \hU(h(\hU-1,\hV) - h(\hU,\hV)) \right]\notag\\
	&\ \ \ + \E \left[ \sum_{i=1}^n\left[ (1-p_i)s_i(h(\hU,\hV+1) - h(\hU,\hV))\right] + \hV(h(\hU,\hV-1) - h(\hU,\hV)) \right].\label{ex1}
\end{align}
We begin with
\begin{align*}
\E& \left[  \hU(h(\hU-1,\hV) - h(\hU,\hV)) \right] = \E \sumin \hU_i(h(\hU-1,\hV) - h(\hU,\hV))\\
&=  \sumin\E \left[ \hU_i(h(\hU-1,\hV) - h(\hU,\hV)) \middle| \hU_i=1, \hV_i = 0\right] \Pr(\hU_i = 1, \hV_i = 0)\\
&= \sumin  p_ir_i \E \left[ h(\hUi, \hVi) - h(\hUi + 1, \hVi) \right] ,
\end{align*}
where $\hUi = \hU - \hU_i$ and $\hVi = \hV - \hV_i$. Hence the first half of~\Ref{ex1} becomes
\begin{align}
\sumin p_ir_i \E \left[ (h(\hU+1,\hV) - h(\hU,\hV)) +  (h(\hUi, \hVi) - h(\hUi + 1, \hVi)) \right]. \label{ex1a}
\end{align}
We now consider three cases: (i) $U_i = 0, V_i = 0$, (ii) $U_i = 1, V_i = 0$, and (iii) $U_i = 0, V_i = 1$. The second case can be termed a false negative, and the third, a false positive.  Note that it is impossible for an edge to be a false positive and false negative at the same time. In the first of these three cases, the terms in~\Ref{ex1a} will cancel out to 0, and in the latter two cases we get exactly a second difference of the function $h$, and these two cases take probability $p_i r_i$ and $(1-p_i)s_i$ respectively. Therefore, \Ref{ex1a} can be bounded by
\begin{align}
\|\Delta^2_{ij} h\| \cdot \sumin p_ir_i(p_ir_i + (1-p_i)s_i).\label{ex1b}
\end{align}
An analogous argument follows for the second half of~\Ref{ex1}, and therefore the entirety of~\Ref{ex1} can be bounded by
\begin{align*}
\|\Delta^2_{ij} h\| \cdot \sumin (p_ir_i + (1-p_i)s_i)^2,
\end{align*}
and the final bound follows from Theorem~\ref{steinfactorthm} and Remark~\ref{plusbound}.
\end{proof}

As a simplification to aid with interpretation of the bound, if we set $p_i = p, r_i =r$ and $s_i = s$, the bound becomes
\begin{align*}
\frac{2}{n} + (pr + (1-p)s) \cdot 2\sqrt2\log\left(\sqrt 2n(pr + (1-p)s)\right).
\end{align*}
The assumption that the error probabilities $r_i$ and $s_i$ are constant across the graph is referred to as a homogeneity assumption in~\cite{BKV14}.  While likely not strictly true in practice, it is a useful assumption for better illustrating how the relevant aspects of the problem combine to influence the accuracy of approximation by Skellam.  If we further assume that $\lambda_1$ and $\lambda_2$ are equal to some common value, say $\lambda$, our setup is then roughly equivalent to that in~\cite{BKV14}.  This assumption can be viewed as imposing a type of centering on the noise at the level of individual edges, since it dictates that in expectation we have $|E|$ equal to $|\hat{E}|$.  In this case, since $rp = s(1-p) = \lambda/n$, the bound becomes
\begin{align*}
\frac{1}{n}\left[ 2 + 4\sqrt2\lambda \log\left(2\sqrt2\lambda\right)\right] .
\end{align*}
When it is not unreasonable to expect that $\lambda$ vary with $n$, we then find that the accuracy of approximation by Skellam in this problem - for this special case - varies like $O\left(\lambda_n \log(\lambda_n) / n\right)$.


The method of proof of this bound is unsurprisingly similar to Poisson approximation of the sum of independent but not necessarily identical Bernoulli trials. In our case, there are essentially three components of error terms that we would expect to appear: two of them will result from the individual Poisson approximations of $\hU$ and $\hV$ and then there should be a third term which deals with the fact that $\hU$ and $\hV$ are not independent. For readers familiar with Poisson approximation, you can see where the `third' component of the error appears in~\Ref{ex1b}. The difference arises because the conditioning we make upon $U_i$ has ramifications on $V_i$ as they are not independent. One would expect a single univariate Poisson approximation would only have a sum of the $p_i^2r_i^2$ in ~\Ref{ex1b}, but we require the second term in our scenario. However in some sense, this extra term disappears in the final bound because our Stein factor has both $\l_1$ and $\l_2$ in the denominator. 


\subsection{Haar wavelet coefficients under photon imaging with spillover effects}

Current state of the art in high-quality imaging applications, such as are encountered in medicine and scientific research, makes heavy use of what is known as a charge-coupled device (CCD).  A CCD converts electrical charges to digital values.  In the context of imaging, these electrical charges in turn derive from the conversion of photons -- essentially, particles of light -- into an electrical signal.  Therefore, CCDs (and a variety of other related devices) are central to modern image acquisition and digital image processing, in that by assembling arrays of CCDs and orienting them towards an object of interest it is possible to represent that object through a matrix of photon counts over the individual CCDs in the array.

Ideally, the count in each CCD would be independent of the others and relevant only to a certain corresponding portion of the imaged object.  However, for technical reasons, there can be various types of degradation.  For example, it typically is the case that photons that should be counted in a given CCD actually can be counted in others.  This effect is sometimes referred to as `spillover' and can be thought of as inducing a type of blurring in the image.  Standard practice is to calibrate imaging instruments before use, yielding a (usually) probabilistic mapping function that characterizes the blurring.   Depending on the extent of such degradation and the application at hand, this may be used in turn for deblurring in the image processing stage.  See, for example, \cite{BKS13,MFRS13} for a detailed description of this paradigm in the context of X-ray imaging in astrophysics.  

Here we set up a simple caricature of the type of image degradation problem just described, in which a weak dependence among photon counts results.  Without loss of generality, we consider a one-dimensional signal rather than a two-dimensional image.  In practice, the indexing in this dimension is typically photon energy, rather than photon source location.  But the same types of degradation issues can be present.  For our signal processing, we consider the use of wavelets, a work-horse in signal and image processing for over $20$ years now~\cite{M08}.  Specifically, both for simplicity and to match most closely the focus of this paper, we consider the use of the Haar wavelet.  The result of applying a Haar wavelet transform to a one-dimensional signal is to produce a collection of Haar coefficients which, as the inner product of the wavelet and the signal, are proportional to the difference of the sums of the signal values over two adjacent windows.  

Suppose we had $n$ bins (e.g., corresponding to CCDs), and note in the following that all defined vectors will be of length $n$. Let the vector $\tX$ be the true signal and suppose $\tX \ed \Po(\tf)$, so the $X_i \ed \Po(f_i)$ and are also independent from each other. It has been shown~\cite{HW12} that both the wavelet and scaling coefficients for the Haar wavelet are distributed as (proportional to) Skellam random variables with parameters comprised of sums and differences of the elements of $\tf$. Set $\tP$ where $P_i \in \{0,1\}$ denote the positive inclusions for a given Haar wavelet coefficient, similarly $\tN$ with $N_i \in \{0,1\}$, for the negative inclusions and $\tP + \tN = \{0,1\}^n$, that is there is no overlap of 1's. Then the Haar wavelet coefficient can be represented as $U-V$ where $U = \tP \cdot \tX$, $V = \tN \cdot \tX$ and $\cdot$ denotes the dot product. Furthermore, $U-V \ed \Sk(\tP \cdot \tf, \tN \cdot \tf)$. In the following we will investigate how measurement errors would impact the distribution of these coefficients.

A simple variant of the type of spillover referred to above, in the context of a one-dimensional signal, is when a particle may actually end up being observed at a lower energy level than its true energy. In our model we will assume that  each particle that arrives is independent and there is a probability $p$ that the particle will be observed in exactly one level lower than its true energy. Let $Y_i$ denote the number of particles in bin $i$ that were observed correctly, and $Y_i^*$ denote the number of particles in bin $i$ that were the result of errors in measurement. That is $Y_i^*$ is the number of particles of energy level $i+1$ but were measured at level $i$.

Due to the thinning property of Poisson random variables, $\tY$ and $\tY^*$ are independent. Set $U' = \tP \cdot \tY + \tP \cdot \tY^*$ and $V' = \tN \cdot \tY + \tN \cdot \tY^*$.  The observed Haar wavelet coefficient satisfies
\[ U'-V' \ed \text{Sk} \left( (1-p)\tP \cdot \tf + p \tP\cdot \tf^{(-1)}, (1-p)\tN \cdot \tf + p \tN \cdot \tf^{(-1)} \right),\]
where $f^{(-1)}_i = f_{i+1}$. Note that we can set $P_{-1}=N_{-1}=0$ and $f_{n+1}=0$ to avoid boundary issues. So our question is, what is the difference between these two different Skellam distributions, i.e., between the distributions of the true and observed Haar wavelet coefficients.

\begin{theorem}
In the above set up,
\[ d_{TV}(\mathcal{L}(U'-V'), \mathcal{L}(U-V)) \leq \sqrt{\frac{2p^2}{e \max(\tP\cdot\tf, \tN\cdot\tf)}} \left[ | \tP \cdot \tf - \tP \cdot \tf^{(-1)}| + | \tN \cdot \tf - \tN \cdot \tf^{(-1)}|\right]. \]
\end{theorem}
\begin{proof}
To bound this difference in total variation, we use a simple adaptation of Theorem~1.C part (i) from~\cite{BHJ92}. Using the true distribution of $U-V$ as our `reference' measure, we need to bound $| \E \mathcal A(U',V')|$ from~\Ref{2dimgen} where $\l_1 = \tP \cdot \tf$ and $\l_2 = \tN \cdot \tf$. Note that using the usual Poisson Stein identity,
\[ \E\left[ U' [ h(U'-1, V') - h(U',V')] | V'\right] =  -\left((1-p)\tP \cdot \tf + p \tP\cdot \tf^{(-1)}\right)\E [h(U'+1,V') - h(U',V')| V'],\]
therefore to bound the first half of~\Ref{2dimgen}, 
\begin{align*}
|\E &\left[ \tP \cdot \tf  [h(U'+1,V') - h(U',V')]  +  U' [ h(U'-1, V') - h(U',V')]\right]|\\
	& \  \ \ = |\left((\tP - (1-p)\tP) \cdot \tf - p \tP\cdot \tf^{(-1)}\right) \E[\E [h(U'+1,V') - h(U',V')|V']]|\\
	&\ \ \ \leq \|\Delta_i h\| \ p\ |\tP\cdot \tf - \tP \cdot \tf^{(-1)}|. 
\end{align*}
An analogous bound can be derived for the second half~\Ref{2dimgen} and this yields the final result.
\end{proof}

Note that the bound in the above theorem is larger when, relative to the larger of total signal intensity in the positive or negative window (i.e., the larger of $\tP\cdot\tf$ or $\tN\cdot\tf$), the discrepancy in those totals resulting from a shift of the windows by one is large.  That is, when the windows are near a spike or jump in the underlying signal $\tf$.  Therefore, in particular, the effects of spillover are minimal in regions of the signal that are smooth.

\begin{remark}
If we wished to generalise this result to allow the error probability to be random, for example the error rate for bin $i$ could depend upon $X_i$, this should in theory be possible by adapting Theorem~1.C part (ii) from~\cite{BHJ92}.
\end{remark}

\section{Discussion}
There is one notable drawback in the approach used in this paper. Given our approach is to project from two dimensions to one using appropriate test functions, this will only be applicable when approximating the difference of two random variables. If one wishes to approximate a single univariate random variable with the Skellam distribution directly, then this approach will not be useful. It remains open whether a direct one dimensional approach is possible.

Poisson approximation via the generator method involves characterising the Stein identity as the generator of an immigration-death Markov process where the immigration rate is constant $\lambda$ and the death rate is unit per capita. Such a generator characterises the Poisson distribution as it is the unique stationary distribution of such a process. Intuitively, for the Skellam distribution one would aim to construct a generator defined on the integers such that $\l_1$ would denote the rate of increase of `positive particles', $\l_2$ the rate of increase of `negative particles', and then an offsetting death-type rate that would remove particles appropriately, thus ensuring the process does not explode in either direction so that the associated stationary distribution is Skellam. The problem with attempting such a construction, from a one dimensional viewpoint, is that if we only knew the difference between the two counts of positive and negative particles, this is not enough information to properly define the transition rates of the process. For example, if we knew that the difference of the two counts was 0, there are infinitely many possibilities for the number of positive and negative particles, and to properly define the process we need to know how many positive and negative particles there are. The problem described above with constructing an appropriate one dimensional generator for the process is what leads us to believe that a one dimensional approach is not possible using the generator method, however we concede that it is possible that there may exist a generator representation that would be amenable to analysis.

An interesting question is whether there exists a nice clean bound for the first difference of $h$ of the order $\frac{1}{\sqrt{\l_1 + \l_2}}$ as opposed to our two bounds in \Ref{firstdiff}, \Ref{firstdiff2}. Our bound~\Ref{firstdiff2} is derived via the inversion formula for characteristic functions. The `usual' method that is used in Poisson approximation does not seem viable in the Skellam scenario, primarily because it involves finding a uniform bound for the maximum of the Poisson mass function in terms of $\l$. For the Skellam distribution, one might suspect an analogous approach, however given we have one quantity to bound but two parameters to work with, this method seems unfruitful. We expect that it should be possible to find such a bound, and this remains an interesting open problem to solve.

Similarly for the second difference our bound involving the maximum of $\l_1$ and $\l_2$ should be able to have all the maximum terms replaced with the sum of the two parameters without the penalty invoked in Remark~\ref{plusbound}. The correct way to derive such a bound would be to condition upon the difference $Z_0^{\l_1}(t)- Z_0^{\l_2}(t)$ in~\Ref{part2} rather than just one of the two processes. This would ultimately require a bound upon the sum of the absolute second differences of Skellam probabilities. For $p_k = \Pr(Z=k)$ where $Z \ed \text{Sk}(\l_1,\l_2)$, numerical results indicate that $\sum_k |p_k - 2p_{k-1} + p_k| \leq \frac{1}{\l_1 + \l_2}$, which intuitively makes sense given the Poisson bound, as both $\l_1$ and $\l_2$ will `flatten' out the mass function as they increase. However we were unable to prove such a result, as the Bessel functions proved to be not very tractable.  

It is worth comparing our bound for the second difference to the Stein solution~\Ref{seconddiff} to the corresponding bounds derived in Theorem~4 in~\cite{BKV14}, where it was shown $\|\Delta^2_{11} h \| \leq \frac{160}{2\l}$, but limited to the case where $\l = \l_1 = \l_2$ and for the Kolmogorov metric. However, given the test functions for total variation distance and Kolmogorov distance are not completely dissimilar, a comparison is still worthwhile. Our bound in this paper is of order $\frac{\log(\l_1 + \l_2)}{\l_1 + \l_2}$, so for very large $\l_1, \l_2$ this will fare worse. However our constant is much better so this will only be worse on very large $\l_1, \l_2$. And, obviously, our bounds have the significant added flexibility of not requiring $\l_1 = \l_2$.

In light of the bounds of order $\frac{1}{\l_1 + \l_2}$ in~\cite{BKV14}, an interesting question is whether our bounds in this paper of $\frac{\log(\l_1 + \l_2)}{\l_1 + \l_2}$ are of the right order. Given that \cite{B05} has shown that for multivariate Poisson approximation the Stein factors are of strict order $\frac{\log(\l_1 + \l_2)}{\l_1 + \l_2}$, and our approach involves adapting bivariate Poisson approximation on specific test functions we believe that our order may be the best possible using our approach. 

Using the generator approach, the standard bound for the Poisson Stein factor, see Corollary~2.12 of~\cite{BX01} for example, involves a coupling based upon hitting times of an immigration-death process. However this coupling is difficult to use in the multivariate case as hitting times become significantly more complicated when there are multiple dimensions. Logarithmic terms are quite common in Poisson related approximation theory, such as multivariate Poisson as discussed above, and also for process approximation where it has been shown that logarithmic terms are strictly necessary if we wish to use uniform bounds for the Stein factors~\cite{BX95}. Our approach in this paper has both aspects of multivariate and univariate analysis, multivariate in the sense that we are essentially considering a special case of bivariate Poisson approximation, but the ultimate target is Skellam which is univariate. As a result, it is not clear what the correct order should be. We would lean towards the correct order not including a logarithmic term, but such a bound is likely beyond the methods used in this paper. Whether a direct analytic, or alternative approach would yield a better result is unknown.

\subsection*{Acknowledgments}
We would like to thank two anonymous referees for their careful reading of this manuscript and their helpful comments and suggestions. This work was supported in part by AFOSR award 12RSL042. We would also like to thank Nathan Ross for pointing out a slightly sharper bound for the second differences of Poisson probabilities used in the proof for Theorem~\ref{steinfactorthm}.

\bibliographystyle{apt}
\bibliography{Skellam}{}

\end{document}